\documentclass[11pt]{article}

\usepackage[utf8]{inputenc}
\usepackage[english]{babel}
\usepackage{amsthm}
 \usepackage{amssymb}
 \usepackage{amsmath}
 \usepackage{hyperref}

\usepackage{comment}

\usepackage{graphicx}
\usepackage{color}
\usepackage{tikz}
\usepackage{float}
\usepackage{subfig}

\newtheorem{thm}{Theorem}[section]

\newtheorem{lem}[thm]{Lemma}

\newtheorem{remark}[thm]{Remark}

 \newcommand{\RR}{\mathbb{R}}
 \newcommand{\CC}{{\mathbb C}}
 \newcommand{\NN}{{\mathbb N}}

\newcommand{\lcal}{\mathcal{L}}

  \def\calL{\lcal}

 \def\be{\beta}
\def\o{\omega} 
\def\th{\theta} 
\def\vp{\varphi} \def\eps{\epsilon}

\newcommand{\dbar}{\bar\partial}

\def\Im{{\operatorname{Im}}}
\def\Re{{\operatorname{Re}}}

\def\KE{K\"ahler--Einstein }

\def\Ric{\hbox{\rm Ric}\,}

\def\h#1{\hbox{#1}}

\def\ra{\rightarrow}
\def\pa{\partial}

\def\w{\wedge}
\def\i{\sqrt{-1}}

\def\be{\beta}

\def\beq{\begin{equation}}
\def\eeq{\end{equation}}
\def\beqno{\begin{equation*}}
\def\eeqno{\end{equation*}}
\def\eaeq{\end{aligned}}
\def\baeq{\begin{aligned}}

\def\bpf{\begin{proof}}
\def\epf{\end{proof}}

\def\eaeq{\end{aligned}}
\def\baeq{\begin{aligned}}

\def\be{\beta}

\def\ra{\rightarrow}
\def\Foot{\operatorname{Foot}}

\begin{document}
\title{
Small angle limits of Hamilton's footballs
}

\date{May 2, 2019}

\author{Yanir A. Rubinstein, Kewei Zhang}
\maketitle

\begin{abstract}
Compact Ricci solitons on surfaces have at most
two cone points, and are known as Hamilton's footballs.
In this note we completely describe the degenerations
of these footballs as one or both of the cone angles
approaches zero. In particular, 
we show that Hamilton's famous non-compact cigar soliton is 
the Gromov--Hausdorff limit of Hamilton's compact conical 
teardrop solitons. 
\end{abstract}

\section{Introduction}

In this note we show that two seemingly very different Ricci
solitons constructed by Hamilton are in fact closely related.
Namely, we show that
 Hamilton's 
non-compact steady Ricci soliton \cite[p. 256]{Ham88}, also known as the cigar soliton, is 
the Gromov--Hausdorff limit of compact shrinking Ricci solitons with a
conical singularity
constructed by Hamilton in the same foundational article \cite[p. 261]{Ham88},
also known as teardrop solitons. 
In fact, the cigar soliton turns out to be
the blow-up limit of the teardrop solitons, and
the limit takes place as the {\it cone angle tends to zero.}
More generally, we describe all possible degenerations of 
Ricci solitons with at most two cone-points as the cone angles
tend, {\it possibly jointly}, to zero.

This fits in nicely and is motivated by a conjectural picture
put forward by Cheltsov and one of us \cite{CR,R14} 
 in which non-compact Calabi--Yau fibrations
emerge as the small angle limit of families of
compact singular Einstein metrics known as \KE edge metrics.
In fact, our result suggests this conjectural picture
should extend to solitons and in a subsequent article
we plan to pursue this \cite{RZ}. Moreover, our result
concretely illustrates the difficulty in treating 
divisors with more than one component and the ensuing joint
small angle asymptotics.

\def\Cigar{\operatorname{Cigar}}
\def\Foot{\operatorname{Foot}}
\def\TD{\operatorname{TD}}
\def\Cyl{\operatorname{Cyl}}

\subsection{Cigar as a limit of teardrops}

As shown by Hamilton, there exists a soliton metric with
a single conical singularity of angle $2\pi\be$ on $S^2$ and area
$2\pi(1+\be)$
and such K\"ahler--Ricci solitons are
nowadays known to be unique in any dimension 
\cite{Berndtsson} (below we will give an alternative uniqueness 
proof in this setting, see Remark
\ref{uniqremark}). We 
denote this metric by $g_{\TD,\be}$ for each $\be\in(0,1)$.
Here, we consider 
$g_{\TD,\be}$ as a tensor on $\RR^2 \cong S^2\setminus\{\h{\rm cone point}\}$.
On the other hand, Hamilton's cigar soliton is the metric
\begin{equation}
\label{cigareq}
g_{\Cigar}=\frac{dx\otimes dx+dy\otimes dy}{1+x^2+y^2}
\end{equation}
on $\RR^2$.

\begin{thm}
\label{mainthm}
The cigar soliton on $\RR^2$ is 
the pointed smooth (and hence also Gromov--Hausdorff) limit of rescaled conic teardrop
solitons on $S^2$.
More precisely, considered as tensors on $\RR^2$,  pointwise in every $C^k$-norm 
$$
g_{\TD,\be}/2\be\;
\buildrel{\be\ra0}
\over{\xrightarrow{\hspace*{1cm}}}
\;
g_{\Cigar}
,
$$
where $g_{\TD,\be}$ is the unique soliton metric of area $2\pi(1+\be)$ with a single
cone point of angle $2\pi\be$ on $S^2$.
\end{thm}

In fact, the proof of Lemma 
\ref{finalconvlem} gives that 
the rate of convergence in Theorem
\ref{mainthm} is
linear in $\be$ in a certain coordinate chart.

\subsection{Degenerations of footballs}

In fact, Theorem \ref{mainthm} is
a rather special case of a more general 
phenomenon that we now describe.

Let us work 
more generally with
football solitons
$
g_{\Foot,\be_1,\be_2}
$
 that allow two cone points,
namely one of angle $2\pi\be_1$ at $N$ (the north pole) and one of
angle $2\pi\be_2$ at $S$ (the south pole). 
On the other hand, let us identify $\mathbb{R}^2$ with $S^2\backslash\{N\}$. The non-compact cone-cigar
soliton of angle $2\pi\be$ at the origin is given, in polar coordinates, by
\beq
\label{conecigareq}
g_{\Cigar,\be}
=
\frac{dr\otimes dr+\be^2 r^2d\theta\otimes d\theta}{1+r^2}
\eeq
(in Remark \ref{whysolitonRemark} we show that this indeed
solves the Ricci soliton equation).
When $\be=0$ we consider
$g_{\Cigar,0}$ as a metric on 
$\RR_+$ (note that indeed the origin is at finite distance
from any point).
Note that $g_{\Cigar}=
g_{\Cigar,1}$ and 
$g_{\TD,\be_1}
=
g_{\Foot,\be_1,1}$, 
so Theorem \ref{mainthm} is the case
$\be_2=1$ in the following:

\begin{thm}
\label{conecigarthm}
The cone-cigar soliton on $\RR^2$ 
or $\RR_+$
(when $\be>0$ or $\be=0$, respectively)
is 
the pointed smooth (and hence also Gromov--Hausdorff) limit of rescaled conic teardrop
solitons on $S^2$.
More precisely, considered as tensors on $\RR^2\setminus\{0\}$,  pointwise in every $C^k$-norm 
$$
\Big(
S^2,
\frac{\be_2}{2\be_1}g_{\Foot,\be_1,\be_2}, S\Big)\;
\buildrel
{
\frac{\be_1}{\be_2}\ra0,
\;\;
\be_2\ra \be
}
\over
{
\xrightarrow{\hspace*{2.7cm}}
}
\;
(\mathbb{R}^2,g_{\Cigar,\be},0),
$$
where $g_{\Foot,\be_1,\be_2}$ is the unique soliton metric on $S^2$ of area $2\pi(\be_1+\be_2)$ with two cone points of angles $2\pi\be_1<2\pi\be_2$ at $N$
and $S$.
\end{thm}

\begin{figure}[H]
\centering
\includegraphics[width=0.7\textwidth]{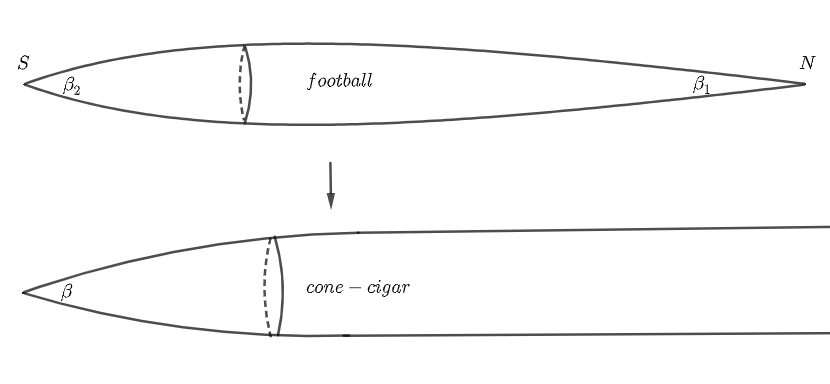}
\end{figure}

Theorem \ref{conecigarthm} describes
{\it joint} degenerations 
of the 
cone angles with one converging
to zero faster than the other
(that may or may not converge to zero itself) 
and obtain a (possibly collapsed) cone-cigar in the jointly rescaled limit.
Finally, we complete the picture by
describing the asymptotic limit
when both angles converge to zero
at comparable speed (which can
be considered as the hardest
case, in a sense). Denote by
$$
g_{\Cyl}:=
\frac{dx\otimes dx+dy\otimes dy}{x^2+y^2},
\quad (x,y)\in\RR^2\setminus\{0\},
$$
the pull-back of the flat metric on $\CC$ to the cylinder $\CC^*=\CC\setminus\{0\}$ under the map $z\mapsto\log z$.

\begin{thm}
\label{cylinderthm}
Let $c\in(0,1)$.
Considered as tensors on $\RR^2\setminus\{0\}$,  pointwise in every $C^k$-norm 
$$
\Big(
\frac{1}{\be_1^2}
g_{\Foot,\be_1,\be_2}, p_{\be_1,\be_2}\Big)\;
\buildrel
{
\frac{\be_1}{\be_2}\ra c,
\;\;
\be_2\ra 0
}
\over
{
\xrightarrow{\hspace*{2.7cm}}
}
\;
(Cg_{\Cyl},e^0),
$$
where $C=C(c)$, and where $g_{\Foot,\be_1,\be_2}$ is the unique soliton metric on $S^2$ of area $2\pi(\be_1+\be_2)$ with two cone points of angles $2\pi\be_1<2\pi\be_2$ at $N$
and $S$, and 
$p_{\be_1,\be_2}$ is the unique point in $\tau^{-1}(\be_1)\subset S^2$
with $\th=0$
and $g_{\Cyl}$ is
the flat cylinder 
metric on $\RR^2\setminus\{0\}$.
\end{thm}
Note that $p_{\be_1,\be_2}$  is the unique point
on the circle $\{\tau=\be_1\}\subset S^2$
with $\th=0$; this $\tau$-level set is characterized by the property 
that the region between the circle and the pole $N$
has area $\be_1$ with respect to $g_{\Foot,\be_1,\be_2}$ (see Lemma \ref{expression of KRS} for more details on the $(\tau,\theta)$ coordinates).

We may summarize Theorems \ref{mainthm}--\ref{cylinderthm}
in a succinct figure. The {\it moduli space of
footballs} can be parametrized by the angle coordinates
$(\be_1,\be_2)\in\RR_+^2\setminus\{0\}$ where we represent each point
$(\be_1,\be_2)$ by the unique Ricci soliton with cone
angles $2\pi\be_1$ at $N$ and $2\pi\be_2$ at $S$ and 
of area $2\pi$ (unlike the normalization in the theorems above!).
Then to describe the asymptotic behavior near the boundary
of $\RR^2_+$ it is most natural to blow-up the origin
in $\RR^2_+$ and use the coordinates 
$(r,\th_1):=(\sqrt{\be_1^2+\be_2^2},\be_1/\be_2)$
or $(r,\th_2):=(\sqrt{\be_1^2+\be_2^2},\be_2/\be_1)$:

\begin{figure}[H]
\centering
\includegraphics[width=0.7\textwidth]{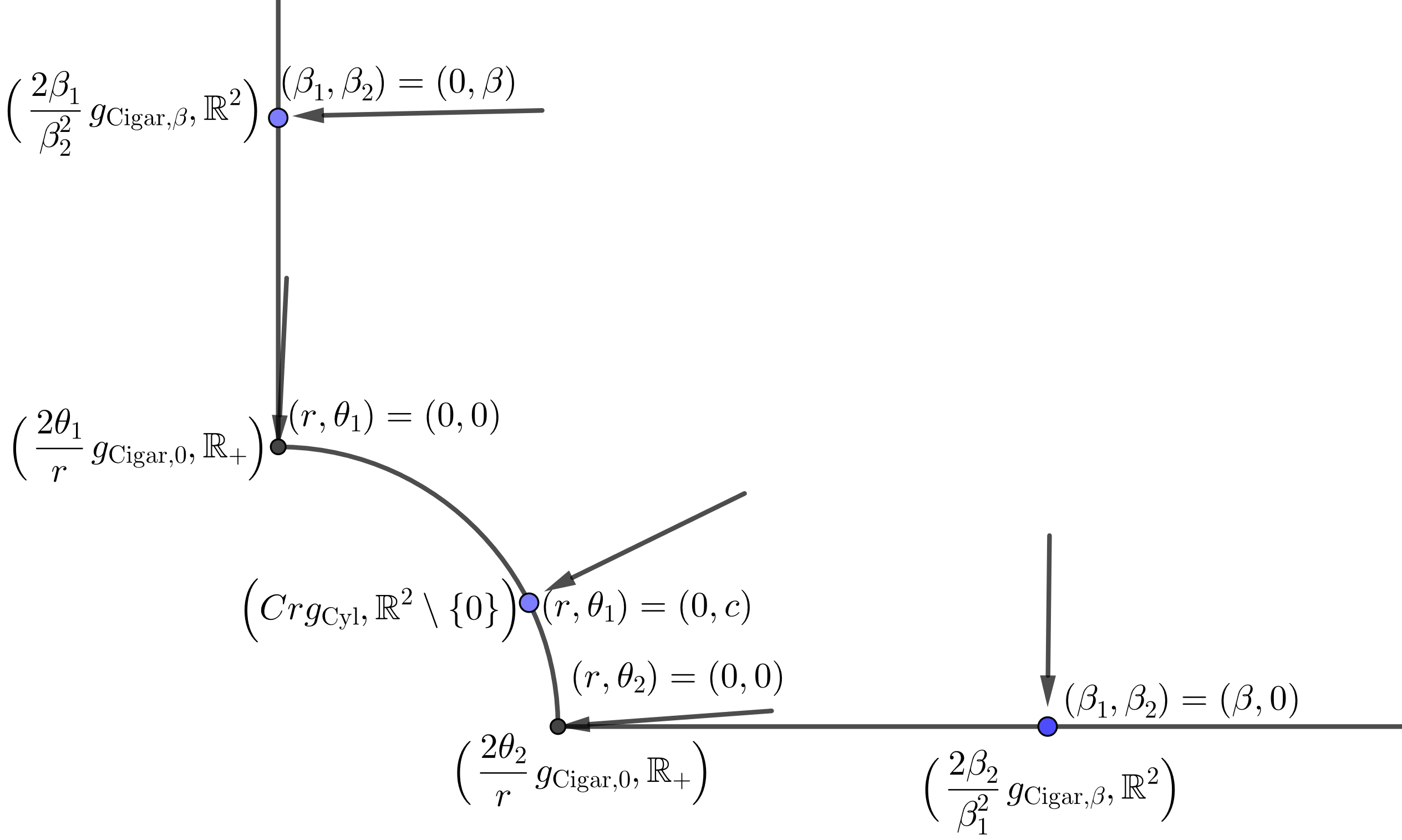}
\end{figure}

\noindent
For example,  
in the joint limit $\be_1/\be_2\ra 0, \; \be_2\ra0$,
Theorem
\ref{conecigarthm} with $\be=0$ translates to
the area $2\pi(1+\be_1/\be_2)\approx 2\pi$
metric 
$
\frac{g_{\Foot,\be_1,\be_2}}{\be_2}
$
being asymptotic to 
$
\frac{2\be_1}{\be^2_2}g_{\Cigar,0}
$
which in the projective coordinates
$(r,\th_1)$
is asymptotic to 
$
\frac{2\theta_1}{r}g_{\Cigar,0}.
$
The other arrows in the figure are 
obtained similarly from Theorems  \ref{mainthm}--\ref{cylinderthm}.
In the language of \cite[(1.2)]{CR2} this gives a complete 
description of the boundary of the body of ample angles
of the pair $(S^2,N+S)$. A very intersting open problem is to
generalize this to other pairs.

In the next section we begin by recalling Hamilton's construction of the conical
teardrop and football solitons, using slightly different language than his, namely
the by-now-standard moment map picture going back to Calabi. 
The
elementary proofs of Theorems \ref{mainthm}--\ref{conecigarthm}
then follow by an asymptotic analysis of the resulting ordinary differential equations.
In the final section we extend these arguments to the more difficult
case of angles tending to zero with comparable speeds and prove
Theorem \ref{cylinderthm}.

\section{Small angle asymptotics of footballs}
\label{sec:ODE}

In this section we prove Theorem \ref{conecigarthm} (that contains Theorem \ref{mainthm}
as a particular case).
We start by  recalling some of the most pertinent details of the construction
of the teardrop soliton on the unit sphere, due to Hamilton \cite{Ham88} (cf. Ramos \cite{Ramos}) but 
using the Calabi ansatz approach instead 
\cite{Calabi} (see also, e.g., \cite{HwangS}).

The Ricci soliton equation with two conical singularities is 
\begin{equation}
\label{RicEq}
\Ric \o_{\be_1,\be_2}=
\o_{\be_1,\be_2}
-\calL_X\o_{\be_1,\be_2}
+2\pi(1-\beta_1)\delta_N+2\pi(1-\beta_2)\delta_S,
\end{equation}
where $\o_{\be_1,\be_2}
(\,\cdot\,,\,\cdot\,)=
g_{\Foot,\be_1,\be_2}(J\,\cdot\,,\,\cdot\,)$ is the associated
volume 2-form (where $J$ is the complex
structure on $S^2$ considered as the Riemann sphere),
$X$ is the soliton vector field 
(that will be determined later) and $\delta_p$ denotes the Dirac
delta at $p$.
In particular, applying cohomological considerations to this equation determines the 
area
\begin{equation}\label{cohomology class of KRS}
\text{Vol}(S^2,
g_{\Foot,\be_1,\be_2})=2\pi(\beta_1+\beta_2).
\end{equation}

The starting point is that the conical soliton metrics $g_{\Foot,\be_1,\be_2}$
(as, clearly, are also the $\th$-indepedent $g_{\Cigar,\be_2}$)
are rotationally symmetric \cite[p. 258]{Ham88},\cite[Lemma 3]{Ramos}. 
To see this one typically starts from the Riemannian definition
of a Ricci soliton, on the smooth locus, 
as a solution of $2Kg=\Ric g = g -\nabla^2 h$
from which it readily follows that $J\nabla h$ is a Killing vector for $g$,
hence induces an $S^1$-symmetry. This then yields \eqref{RicEq}
with $X=\nabla h$.
By the $S^1$-invariance, we see that, on $S^2\backslash\{N,S\}$, the volume 2-form $\o_{\be_1,\be_2}$ can be given by a potential function that only depends on $|z|$. 
Moreover, $\calL_X
\o_{\be_1,\be_2}
=
\i\partial\bar{\partial}h$ with $h$ depending only on $|z|$ as well. So we set 
\beq
\label{slogzeq}
s:=\log|z|^2,
\eeq
and write
\beq
\label{omegabetaeq}
\o_{\be_1,\be_2}=\sqrt{-1}\partial\bar{\partial}f(s),
\eeq
where $f$ is a function to be determined (both
$f$ and $h$ depend on $\be_1,\be_2$ but we omit that
from the notation).

In the following, we work only on the smooth part, namely
$$S^2\backslash\{S,N\}\cong\mathbb{C}^*.$$
On that chart we may use the
holomorphic coordinate $w:=\log z$.
Then, since $s=2\Re\, w$, $\frac\pa{\pa w}=
\frac12(\frac\pa{\pa s/2}-\i\frac\pa{\pa \Im\, w})$, and as $f$ is
independent of $\Im\, w,$
\beq
\label{obetaeq}
\o_{\be_1,\be_2}=
\frac{\partial^2 f}{\pa w \pa \bar w}\i dw\w \overline{dw}=
f^{\prime\prime}(s)\i dw\w \overline{dw}
=
f^{\prime\prime}(s)\frac{\sqrt{-1}dz\wedge 
\overline{d{z}}}{|z|^2},
\eeq
so $f$ must be a strongly convex function. 
Following Calabi, switch to the
moment coordinate
\beq
\label{taueq}
\tau:=f^\prime(s)
,
\eeq
and define a function $\vp$ 
(depending on $\be_1,\be_2$)
on
the image of the gradient of $f$ by 
\beq
\label{vpeq}
\varphi(\tau):=f^{\prime\prime}(s)
\eeq
(simply the inverse of the second derivative of the Legendre transform of $f$).
In the following, we seek an explicit formula for $\varphi(\tau)$ (see (\ref{KRS ODE solution}))
since the expression of $\varphi(\tau)$ can in turn give an explicit formula for the soliton metric. We start by rewriting the
metric in terms of $\tau$.

\begin{lem}\label{expression of KRS}
The restriction of the metric $g_{\Foot,\be_1,\be_2}$ to
$S^2\backslash\{S,N\}$ can be written as
\begin{equation}
\label{gfootmomentumeq}
g_{\Foot,\be_1,\be_2}
=
\frac{1}{2\varphi(\tau)}
d\tau\otimes d\tau
+
2\varphi(\tau)d\theta\otimes d\theta,
\qquad
\tau\in(0,\beta_1+\beta_2),\quad \theta\in[0,2\pi).
\end{equation}
\end{lem}

\begin{proof}
Using \eqref{obetaeq} and standard relations
between Riemaniann, Hermitian, and K\"ahler
metrics, 
$$
g_{\Foot,\be_1,\be_2}
=
\frac{2f^{\prime\prime}(s)}{r^2}
(dr\otimes dr+r^2d\theta\otimes d\theta)$$
where we let $z=re^{i\theta}$.
Recall from \eqref{slogzeq} that
$r=e^{\frac{s}{2}}.$
Hence, 
$$
g_{\Foot,\be_1,\be_2}
=\frac{f^{\prime\prime}(s)}{2}ds\otimes ds+2f^{\prime\prime}(s)d\theta\otimes d\theta.
$$
Using \eqref{taueq}--\eqref{vpeq},
\beq
\label{dsdtau}
ds=\frac{1}{\varphi}d\tau,
\eeq
and using this and \eqref{vpeq} again gives,
$$g_{\Foot,\be_1,\be_2}
=\frac{1}{2\varphi(\tau)}d\tau\otimes d\tau+2\varphi(\tau)d\theta\otimes d\theta.$$
In particular, we see that the volume form of this metric is
$
d\tau\wedge d\theta.
$
Here $\theta\in[0,2\pi)$ and $\tau$ belongs to an interval that we need to determine.
Using \eqref{cohomology class of KRS},
$$\tau\in(c,c+\beta_1+\beta_2),$$
where $c$ is a constant (here is where we really needed the $2\pi$ factors
in \eqref{RicEq}, otherwise we would get $2\pi$ factors in the domain of $\tau$). Note that here, we can simply choose $c$ to be 0. This is because we can add an affine function $-cs$ to $f(s)$ to shift the interval of $\tau$ without changing the metric.
\end{proof}

So in the following, we will use $\tau$ as our variable to search for an explicit expression for $\varphi(\tau)$ and convert the soliton equation to a simple ODE.  
First, recall from \eqref{omegabetaeq} that
\beq
\label{omegavpEq}
\omega_{\be_1,\be_2}=
\i\varphi
dw\w \overline{dw}.
\eeq
Similarly to \eqref{obetaeq}, now using
\eqref{dsdtau}, the Ricci form is given by
\beq
\label{RicomegavpEq}
\begin{aligned}
\Ric\omega_{\be_1,\be_2}
&=-\sqrt{-1}\partial\bar{\partial}\log\varphi
\cr
&=
-\sqrt{-1}\frac{\partial^2\log\varphi}{\pa s^2}
dw\wedge \overline{d{w}}
\cr
&=
-\sqrt{-1}\frac{\partial}{\pa s}
\Big(\frac{\varphi'}{\vp}\frac{d\tau}{ds}\Big)
dw\wedge \overline{d{w}}
\cr
&=
-\i\varphi^{\prime\prime}\frac{d\tau}{ds}
dw\wedge \overline{d{w}}
=
-\i\vp\varphi^{\prime\prime}
dw\wedge \overline{d{w}}.
\end{aligned}
\eeq
In particular, 
the Gaussian curvature is given by
\beq
\label{GaussEq}
K(S^2,\omega_{\be_1,\be_2})=-\varphi^{\prime\prime}(\tau).
\eeq

Now we turn to the potential function $h$ of the soliton vector field. Recall that $h$ only depends on $s$, hence only on $\tau$ and we may write
$h=h(\tau).$
Also recall that $h$ is a holomorphy potential, namely the dual of $\dbar h$ with respect to
Hermitian metric,
$$
\begin{aligned}
\nabla^{1,0}h
&=
\frac{1}{\varphi}
\frac{\partial h(\tau)}{\partial\bar{w}}
\frac{\partial}{\partial w}
\cr
&=
\frac{1}{\varphi}
\frac{\partial h(\tau)}{\partial s}
\frac{\partial}{\partial w}
\cr
&=
\frac{1}{\varphi}h'(\tau)\frac{d\tau}{ds}\frac{\partial}{\partial w}
\cr
&=
h'(\tau)\frac{\partial}{\partial w},
\end{aligned}
$$
is a holomorphic vector field
(which happens to be $X-\i JX$ in 
our earlier notation). This implies
$$
h^\prime(\tau)=a\in\RR,\quad
\tau\in(0,\beta_1+\beta_2)
$$
for some constant $a$ to be determined.
So up to an irrelevant constant, we get
$$
h(\tau)=a\tau.
$$
Hence,
\beq
\label{htaueq}
\begin{aligned}
\sqrt{-1}\partial\bar{\partial}h
&=
\i a\frac{\partial^2\tau}{\pa s^2}
dw\wedge \overline{dw}
\cr
&=
\i a\frac{\partial\vp(\tau)}{\pa s}
dw\wedge \overline{dw}
\cr
&=
\i a\vp'(\tau)\frac{d\tau}{ds}
dw\wedge \overline{dw}
&=
\i a\vp\vp'
dw\wedge \overline{dw}
.
\end{aligned}
\eeq
Combining \eqref{omegavpEq}, \eqref{RicomegavpEq} and \eqref{htaueq}, 
the Ricci soliton equation on the smooth
locus reduces to
\beq
\label{reducedKRSeq}
-\varphi^{\prime\prime}\varphi=\varphi-a\varphi^{\prime}\varphi,
\eeq
and since $\varphi>0$ (by \eqref{vpeq} and the convexity of $f$), 
\begin{equation}\label{KRS ODE}
\varphi^{\prime\prime}(\tau)-a\varphi^\prime(\tau)+1=0.
\end{equation}
Thus the soliton equation becomes an 
ordinary differential equation for $\vp$. 
To solve it, let us determine the boundary conditions. 
\begin{remark}
\label{uniqremark}
{\rm Typically, the boundary conditions
are declared by an ansatz and so one
does not quite obtain uniqueness of 
the teardrop solitons in this method
(instead relying, for the uniqueness on
the general result of Berndtsson 
\cite{Berndtsson}). Here instead, we actually 
prove uniqueness by deriving the boundary
conditions using the asymptotic
expansion of \cite{JMR}.}
\end{remark}

Recall that, by Lemma \ref{expression of KRS},
$\tau\in(0,\beta_1+\beta_2).$
By \eqref{dsdtau} $\tau$ increases as $s$
increases (which in turn increases as $|z|$ does). Thus $\{N\}=\{z=0\}=\{\tau=0\}$
and 
$\{S\}=\{z=\infty\}=\{\tau=\beta_1+\beta_2\}$.
Since $\omega_{\be_1,\be_2}$ 
has cone angle $2\pi\beta_1\in(0,2\pi)$ at $N$, 
it follows from \cite[Theorem 1, Proposition 4.4]{JMR} (as the Ricci soliton equation \eqref {RicEq} is
a complex Monge--Amp\`ere equation
of the form treated in op. cit.)
that $f$ has a complete asymptotic expansion
near $z=0$ whose leading term is $|z|^{2\be_1}$:
$$
\begin{aligned}
\varphi
&\sim C_1+C_2|z|^{2\beta_1}+
(C_3\sin\th+C_4\cos\th)|z|^2+O(|z|^{2+\epsilon})
\cr
&=
C_1+C_2e^{\beta_1s}+
(C_3\sin\th+C_4\cos\th)e^{s}+O(e^{(1+\epsilon)s})
\end{aligned}
$$
(note that $r$ in \cite[(56)]{JMR} is equal to 
$|z|^{\be_1}/\be_1$ in our notation,
see \cite[p. 102]{JMR}).
From \eqref{dsdtau} (for instance) 
we see that $\vp$ must vanish at $N$ (and $S$) since
$\vp>0$ away from $N$ and $S$ and
$\tau$ lives in a bounded interval while $s$
lives on an unbounded one.
Thus, $C_1=0$ (actually also $C_3=C_4=0$ as $\vp$
is independent of $\th$ but we do not need this).
Moreover, the expansion can be differentiated term-by-term 
as $|z|\rightarrow0$ or $s\ra-\infty$.
As $
\vp'(\tau)=\frac{\pa\vp}{\pa s}\frac{ds}{d\tau}
=\frac{\pa\vp}{\pa s}/\varphi,
$
we obtain
\begin{equation}\label{eq:0}
    \varphi(0)=0,\quad \varphi^\prime(0)=\beta_1.
\end{equation}
The same arguments imply that
\begin{equation}
    \label{eq:b1+b2}
    \varphi(\beta_1+\beta_2)=0,
    \quad \varphi^\prime(\beta_1+\beta_2)=-\beta_2.
\end{equation}

Next, we claim that $\beta_1$ and $\beta_2$  determine $a$. Indeed, (\ref{KRS ODE}) is
a first-order equation for $\tau$ and
integrating it yields
\beq
\label{apositive}
a\varphi^\prime(\tau)=Ce^{a\tau}+1.
\eeq
Using the boundary conditions we find
\beq
\label{Cabetaeq}
-(a\beta_2+1)e^{-a(\beta_1+\beta_2)}=
C=a\be_1-1,
\eeq
i.e.,
\begin{equation}
    \label{a_beta}
    a\beta_1-1+(a\be_2+1)e^{-a(\beta_1+\be_2)}=0.
\end{equation}
As we will now show, this can be used to
determine $a$ uniquely from $\be_1,\be_2$,
and, moreover, determines the asymptotic behavior
of $a$ as $\be_1/\be_2\ra 0$.

\begin{lem}
\label{abetaLem}
There is a unique $a=a(\be_1,\be_2)\in(0,\frac{1}{\beta_1})$ solving \eqref{a_beta}. Moreover, 
\beq
\label{beta1alimiteq}
\lim_{\frac{\be_1}{\be_2}\ra0}a(\be_1,\be_2)\beta_1=1.
\eeq
\end{lem}

\begin{proof}
Note that $a\not=0$. Indeed, $a=0$ trivially satisfies \eqref{a_beta} but
then the soliton vector field vanishes and by \eqref{RicEq} we have a metric of constant
scalar curvature which forces $\be_1=\be_2$
\cite[p. 261]{Ham88},\cite[Theorem I]{Tr1989}, contrary to our assumption
that $\be_1<\be_2$.

Put 
\beq
\label{Feq}
F_{\be_1,\be_2}(x)\equiv 
F(x):=\beta_1x-1+(x\be_2+1)e^{-x(\be_1+\beta_2)}. 
\eeq
Compute,$$
F^\prime(x)
=
\beta_1-\big(\beta_1+x\be_2(\be_1+\beta_2)\big)e^{-x(\be_1+\beta_2)},
$$
Notice that $F(0)=F^\prime(0)=0$ and $F$ is asymptotically
linear with slope $\be_1$, and
$\lim_{x\ra\infty}F'(x)=\beta_1$.
Next,
$$
F^{\prime\prime}(x)
=
\Big[
(\be_1+\beta_2)
\big(\beta_1+x\be_2(\be_1+\beta_2)\big)
-
\be_2(\be_1+\beta_2)
\Big]
e^{-x(\be_1+\beta_2)},
$$
so $F''$ is initially negative (as $\be_1<\be_2$) and changes sign precisely
once with $F''\Big(\frac{\be_2-\be_1}{\be_2(\be_1+\beta_2)}\Big)=0$.
Thus, $F'$ vanishes for precisely one positive value $x_0$ of $x$
that is
 a local minimum for $F$ with $F(x_0)<0$. Also, $F$ vanishes for precisely one positive 
value $x_1=x_1(\be_1,\be_2)$ of $x$ and $x_1>x_0$. Note $F(x_1)=0$ means 
$$
\beta_1x_1-1=-(x_1\be_2+1)e^{-x_1(\be_1+\beta_2)}<0
,
$$
i.e., $x_1\in(0,1/\be_1)$ as claimed.

\begin{figure}[H]%
    \centering
    \subfloat[$F^\prime(x)$]{{\includegraphics[width=7cm]{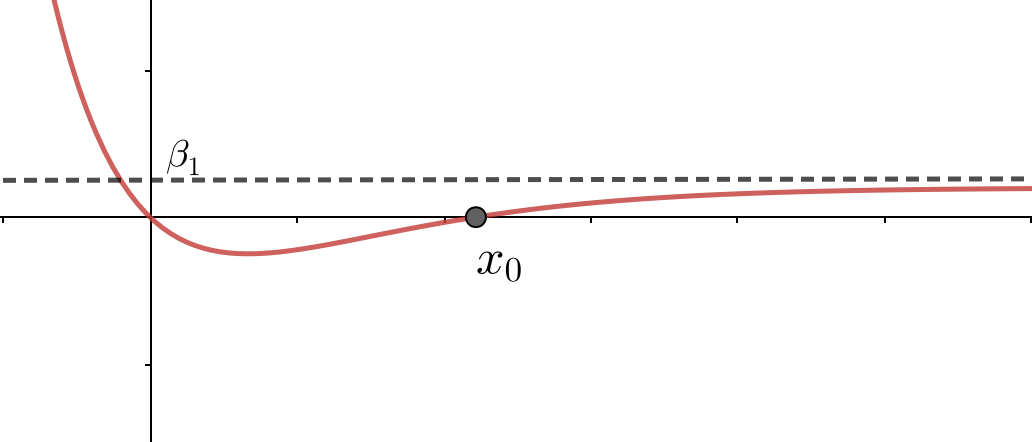} }}%
    \qquad
    \subfloat[$F(x)$]{{\includegraphics[width=7cm]{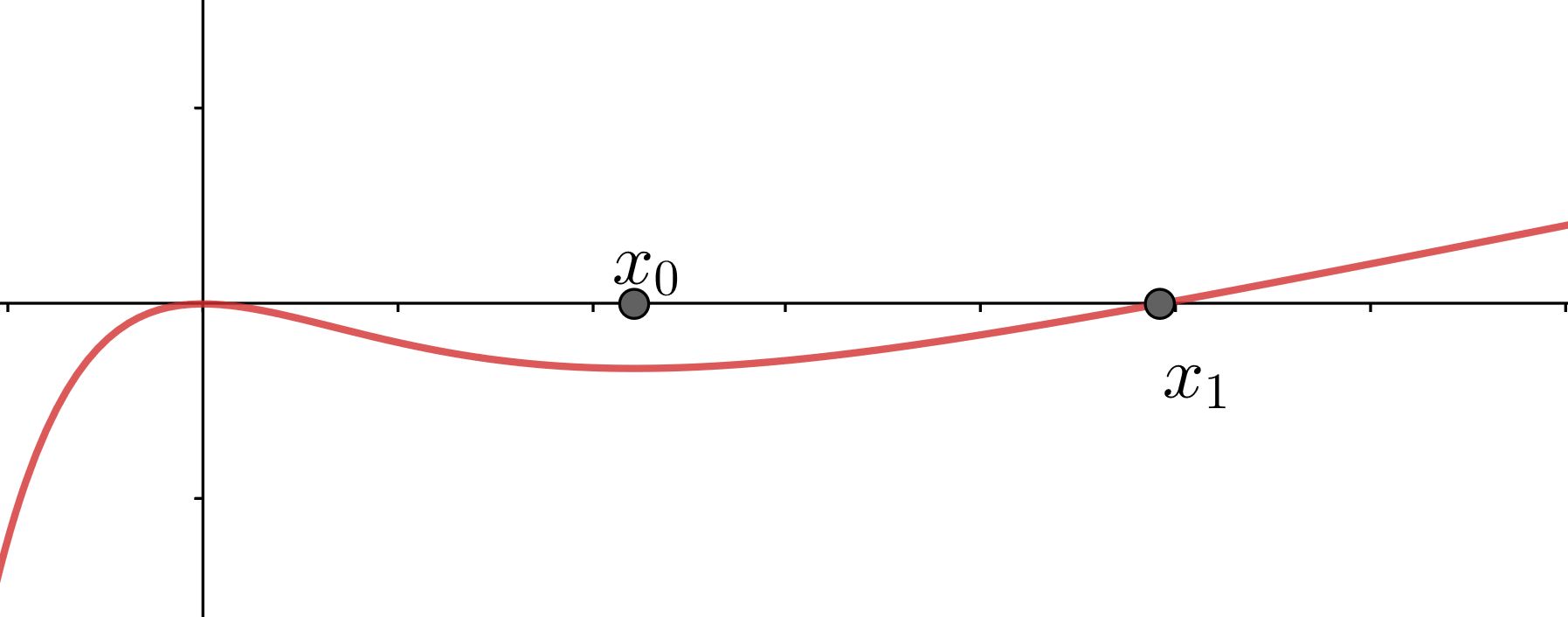} }}%
    \label{fig:example}%
\end{figure}

Finally, 
fix $\eps\in(0,1)$. Then,
\begin{equation*}
    \begin{aligned}
    F\big((1-\eps)/\beta_1\big)
    &=
    -\eps+\big((1-\eps)\be_2/\beta_1+1\big)e^{-(1-\eps)(\be_1+\beta_2)/\beta_1}\\
    &\leq-\eps+\big((1-\eps)\be_2/\be_1+1\big)e^{-(1-\eps)\be_2/\beta_1}.
    \end{aligned}
\end{equation*}
As $\lim_{y\rightarrow\infty}(y+1)e^{-y}=0$ it follows
that $F\big((1-\eps)/\beta_1\big)<0$ for sufficiently small $\beta_1/\be_2$.
Thus, $x_1=x_1(\be_1,\be_2)\in\Big( \frac{1-\eps}{\beta_1}, \frac1{\be_1} \Big)$
for sufficiently small $\beta_1/\be_2$.
Letting $\eps$ tend to zero shows \eqref{beta1alimiteq}.
\end{proof}

With this asymptotic information
we can now study the limit
in Theorem \ref{conecigarthm}.
Recall from \eqref{GaussEq} that
$K(\tau)=-\varphi^{\prime\prime}(\tau).$
Hence, differentiating
\eqref{apositive}
and using \eqref{Cabetaeq},
$K(\tau)=(1-a\beta_1)e^{a\tau}$ for $\tau\in(0,\be_1+\be_2)$.
In particular, the Gaussian curvature is positive and increasing in $\tau$ 
by 
Lemma \ref{abetaLem}. Thus,
using \eqref{a_beta},
$$
\text{sup}_{\tau\in(0,\be_1+\be_2)}
K(\tau)
=
(1-a\beta_1)e^{a(\be_1+\beta_2)}
=
1+a\be_2.
$$
In other words,
the curvature is close to zero near the tip $N$ and 
will become very large near the bottom of the football and its supremum tends to infinity as $\be_1/\be_2$
tends to zero.

Rescale the football metric to make its curvature uniformly bounded,
$$
\tilde g_{\be_1,\be_2}:=
a{g}_{\Foot,\beta_1,\be_2}/2.
$$
By Lemma \ref{expression of KRS},
$$
\tilde{g}_{\beta_1,\be_2}
=\frac{a}{4\varphi(\tau)}
d\tau\otimes d\tau
+a\varphi(\tau)d\theta\otimes d\theta,
\quad \tau\in(0,\be_1+\beta_2).$$
Next, we consider
the asymptotic behavior of this rescaled metric in balls centered at the south pole $S$. 
To that end, introduce a new variable
$$u:=a(\beta_1+\be_2-\tau),\ u\in[0,a\be_1+a\beta_2].$$
So $\{S\}=\{u=0\}$, and
$$
\tilde{g}_{\beta_1,\be_2}
=\frac{1}{4a\varphi}du\otimes du
+a\varphi d\theta\otimes d\theta.
$$
The point is that 
as $\beta_1$ goes to zero, the 
domain of $u$ approaches $\RR_+$ by Lemma \ref{abetaLem}.

\begin{lem}
\label{finalconvlem}
For any fixed constants $L>0$ and $k\in\NN\cup\{0\}$, 
$$
\frac{a}{\be_2}\varphi(u)
\buildrel{\frac{\be_1}{\be_2}\ra0}
\over{\xrightarrow{\hspace*{0.8cm}}}
\frac{e^u-1}{e^u}$$
uniformly in $C^k([0,L])$ .
\end{lem}
Of course, in the statement we mean that we only consider
$\be_1$ sufficiently small, i.e., such that $L<a(\be_1+\beta_2)$.
\begin{proof}
Using \eqref{apositive}--\eqref{Cabetaeq} and the boundary conditions
\eqref{eq:0}--\eqref{eq:b1+b2}
we can integrate (\ref{KRS ODE}),
obtaining
\begin{equation}
\label{KRS ODE solution}
\varphi(\tau)
=
(a\beta_1-1)(e^{a\tau}-1)/a^2
+\tau/a,\quad \tau\in[0,\be_1+\beta_2].
\end{equation}
Here we are implicitly
using that
$a$ is the unique nonzero solution to \eqref{a_beta}.
Recall,\hfill\break $\tau=\be_1+\be_2-u/a$. Thus,
using \eqref{a_beta} again,
\begin{equation*}
\begin{aligned}
a\varphi
&=
(a\beta_1-1)(e^{a(\be_1+\be_2)-u}-1)/a
+\be_1+\be_2-u/a
\cr
&=
\frac{(a\beta_1-1)e^{a(\be_1+\be_2)}}
{ae^u}+(1+a\be_2)/a-u/a
\cr
&=
-\frac{1+a\be_2}
{ae^u}+(1+a\be_2)/a-u/a
\cr
&=
\frac{1+a\be_2}
{a}\frac{e^u-1}{e^u}-u/a.
\end{aligned}
\end{equation*}
The claim now follows from Lemma \ref{abetaLem}.
\end{proof}

In particular, 
$$
\lim_{\frac{\be_1}{\be_2}\ra0,\;\be_2\ra\be}
\be_2\tilde{g}_{\beta_1,\be_2}
=
\frac{e^u}{4(e^u-1)}du\otimes du+\be^2\frac{e^u-1}{e^u}d\theta\otimes d\theta.
$$
Here 
$u\in\mathbb{R}_+$, $\theta\in[0,2\pi)$.
Notice that, if we let
$u=\log(1+r^2),$
we get
$$
\lim_{\frac{\be_1}{\be_2}\ra0,\;\be_2\ra\be}
\be_2\tilde{g}_{\beta_1,\be_2}
=
\frac{dr\otimes dr+\be^2r^2d\theta\otimes d\theta}{1+r^2}
=g_{\Cigar,\be}
$$
on $\mathbb{R}^2$ if $\be>0$
and on $\RR_+$ if $\be=0$.
By Lemma \ref{finalconvlem}, the convergence of the metric tensors is evidently also in
the $C^{k}$-norm for every $k$.
Convergence in the pointed Gromov--Hausdorff topology is an immediate consequence by considering the teardrop minus the cone point embedded in $\RR^2$ and directly using
the definition \cite[Definition 7.3.10]{Burago}.
This, together with Lemma \ref{abetaLem},
completes the proof of Theorem \ref{conecigarthm}.

\begin{remark}
{\rm
\label{whysolitonRemark}
One quick way to see that $g_{\Cigar,\be}$
is indeed a Ricci soliton is to observe that (up to a factor)
by changing
variable to $u=\log(1+r^2)$ and then to $\tau=\be-\be u$
the metric reduces to the form \eqref{gfootmomentumeq}
with $\vp(\tau)=\be^2(1-e^{(\tau-\be)/\be}), \quad \tau\in (-\infty,\be)$,
and $\vp$ solves the equation $\vp''-a\vp'=0$
with boundary conditions $\vp(\be)=0, \vp'(\be)=-\be$ that
by the same analysis leading to
\eqref{reducedKRSeq} precisely
corresponds to 
$$\Ric \o_{\Cigar,\be}=
2\pi(1-\beta)\delta_S
-\calL_X\o_{\Cigar,\be},
$$
which is the equation for a steady soliton on $S^2\backslash\{N\}$ with a cone
singularity of angle $2\pi\be$ at $S$.
}
\end{remark}

\begin{remark}
{\rm
Based on our analysis one may also treat similarly limits
of other families of the solitons
classified by Bernstein--Mettler
and Ramos \cite{BernM,Ramos},
but for conciseness we leave that to the interested
reader.
}
\end{remark}

\section{Cylinder limits of footballs}

In this section we prove Theorem
\ref{cylinderthm}.

\begin{lem}
\label{abeta2Lem}
There is a unique $a=a(\be_1,\be_2)\in(0,\frac{1}{\beta_1})$ solving \eqref{a_beta}. Moreover, 
$$
\lim_{\frac{\be_1}{\be_2}\ra c, \be_2\ra 0}a(\be_1,\be_2)\beta_1
\in(0,1).
$$
\end{lem}
\begin{proof}
The first statement is contained in Lemma \ref{abetaLem}.
For the second, setting
$x:=a\be_2$, and using \eqref{a_beta},
\begin{equation}
    \label{a_beta2}
    x\beta_1/\be_2-1
    +(x+1)e^{-x(1+\beta_1/\be_2)}=0,
\end{equation}
which in the limit gives,
    $$
    cx-1
    +(x+1)e^{-x(1+c)}=0.$$
Observe that this is precisely
$F_{c,1}(x)=0$ (recall \eqref{Feq}), which according
to the proof of Lemma
\ref{abetaLem} has precisely one
solution $x_0\in(0,1/c)$.
Thus, $\lim a\be_2\in(0,1/c)$,
i.e., $\lim a\be_1\in(0,1)$,
as claimed.
\end{proof}
Denote by
$$
b:=\lim a\be_2,
$$
and set
$$
\tau=\be_1-u/a^2.
$$
From \eqref{KRS ODE solution},
\begin{equation*}
\begin{aligned}
a^2\varphi(u)
&=
(a\beta_1-1)(e^{a\be_1-u/a}-1)
+a\be_1-u/a,\quad 
u\in[-a^2\be_2,a^2\be_1]
\cr
&=
\frac{(a\beta_1-1)e^{a(\be_1+\be_2)}}
{e^{u/a+a\be_2}}
+1-u/a
.
\end{aligned}
\end{equation*}
Taking the limit as in the statement
of Theorem \ref{cylinderthm}
gives the limit is a positive 
{\it constant}
$$
\lim a^2\vp(u)=
-\frac{1+b}{e^b}+1=:B>0.
$$
Thus,
$$
\lim a^2g_{\Foot,\be_1,\be_2}
=
\lim
\Big(
\frac{1}{2a^2\vp(u)}
du\otimes du
+
2a^2\varphi(u)d\theta\otimes d\theta
\Big)
=
\frac1{2B}du\otimes du
+
2Bd\theta\otimes d\theta.
$$
Changing variable once
more
to
$v:=u/(2 B)$
we get convergence to
the $2B$ times flat cylinder
$|d\zeta|^2/|\zeta|^2$ 
with $\zeta=e^{v+\i\th}$.
This concludes the proof 
of Theorem \ref{cylinderthm}
since the basepoint $p_{\be_1,\be_2}$
satisfies $(\tau,\th)=(\be_1,0)$, i.e., $u=0$ and $v=0$,
thus limits to $\zeta=e^0=1\in \CC^*$.

\paragraph{Acknowledgments.}

Research supported by NSF grant DMS-1515703, a UMD--FAPESP Seed Grant, and 
the China Scholarship Council award 201706010020.

\def\bi#1{\bibitem{#1}}

\let\omegaLDthebibliography\thebibliography 
\renewcommand\thebibliography[1]{
  \omegaLDthebibliography{#1}
  \setlength{\parskip}{1pt}
  \setlength{\itemsep}{1pt plus 0.3ex}
}

\vspace{0.1in}
 \noindent {\sc University of Maryland}\\
 {\tt yanir@umd.edu}

\bigskip
 \noindent {\sc Peking University \& University of Maryland}\\
 {\tt kwzhang@pku.cn.edu, kwzhang@umd.edu}


\begin{thebibliography}{1}

\bi{Berndtsson}B. Berndtsson,
A Brunn--Minkowski type inequality for Fano manifolds 
and some uniqueness theorems in K\"ahler geometry,
Invent. Math. 200 (2015), 149--200.

\bi{BernM}
J. Bernstein, T. Mettler, Two-Dimensional Gradient Ricci Solitons Revisited, Int.
Math. Res. Notices 2015
(2015), 78–98.

\bi{Burago}D. Burago, Y. Burago, S. Ivanov, A course in metric geometry,
Amer. Math. Soc., 2001.

\bi{Calabi}
E. Calabi: M\'etriques K\"ahl\'eriennes et fibr\'es 
holomorphes, Ann. scient. 
\'Ec. Norm. Sup. 12 (1979), 269--294.

\bi{CR}I.A.~Cheltsov, Y.A.~Rubinstein, {Asymptotically log Fano
varieties}, Adv. Math. 285 (2015), 1241--1300.

\bibitem{CR2}
I.A. Cheltsov, Y.A. Rubinstein,  
On flops and canonical metrics, 
Ann. Sc. Norm. Super. Pisa Cl. Sci. (5) 18 (2018), 283--311. 


	\bibitem{Ham88} R.S. Hamilton, Ricci flow on surfaces, in: Mathematics and General Relativity, Contemporary Math. 71 (1988), 237--261.

\bi{HwangS}
A.D. Hwang, M.A. Singer, 
A momentum construction for circle-invariant K\"ahler metrics, Trans. Amer. Math. Soc. 354 (2002), 2285--2325. 

\bi{JMR}T. Jeffres, R. Mazzeo, Y.A. Rubinstein,
K\"ahler--Einstein metrics with edge singularities, (with an appendix
by C. Li and Y.A. Rubinstein), Annals of Math. 183 (2016), 95--176.

\bi{Ramos}D. Ramos,
Ricci flow on cone surfaces,
Port. Math. 75 (2018), 11--65.

\bibitem{R14}Y.A. Rubinstein, 
Smooth and singular K\"ahler--Einstein metrics, 
in: Geometric and Spectral Analysis (P. Albin et al., Eds.), Contemp. Math.~630, Amer. Math. Soc. and Centre de Recherches Math\'ematiques, 2014, 45--138. 

\bibitem{RZ}Y.A. Rubinstein, K. Zhang, 
Small angle limits of canonical K\"ahler edge metrics,
in preparation.

\bi{Tr1989}
M. Troyanov, Metrics of constant curvature on a sphere with two
conical singularities,
Lect. Notes Math. 1410
(1989),  296--308. 


\end{thebibliography}
\end{document}